\documentclass[a4paper, 11pt]{article}

\usepackage{amsmath}
\usepackage{amsfonts}
\usepackage{amssymb}
\usepackage[english]{babel}
\usepackage{graphicx}
\usepackage{amsthm}
\usepackage{accents}

\def\Z{\mathbb{Z}}
\newtheorem{proposition}{Proposition}
\newtheorem{theorem}{Theorem}
\newtheorem{lemma}{Lemma}

\newtheorem{corollary}{Corollary}

\begin{document}

\title{On disjoint $(v,k,k-1)$ difference families}

\author{Marco Buratti \thanks{Dipartimento di Matematica e Informatica, Universit\`a di Perugia, via Vanvitelli 1 - 06123 Italy, email: buratti@dmi.unipg.it}}

\maketitle
\begin{abstract}
\noindent
A disjoint $(v,k,k-1)$ difference family in an additive group $G$ is a partition of $G\setminus\{0\}$ 
into sets of size $k$ whose lists of differences cover, altogether, every non-zero element of $G$ exactly $k-1$ times. 
The main purpose of this paper is to get the literature on this topic in order, since some authors seem to be unaware of each other's work. 
We show, for instance, that a couple of heavy constructions recently presented as new, had been 
given in several equivalent forms over the last forty years. We also show that they can be quickly derived 
from a general nearring theory result  which probably passed unnoticed by design theorists and that
we restate and reprove in terms of {\it differences}.
We exploit this result to get an infinite class of disjoint $(v,k,k-1)$ difference families coming from the Fibonacci sequence.
Finally, we will prove that if all prime factors of $v$ are congruent to 1 modulo $k$, then there exists a 
disjoint $(v,k,k-1)$ difference family in every group, even non-abelian, of order $v$.
\end{abstract}

\noindent \small{\bf Keywords:} \scriptsize
disjoint difference family; zero difference balanced function; Frobenius group; Ferrero pair; Pisano period.

\normalsize
\eject
\section{Introduction}

Throughout this paper all groups will be understood finite and written in additive notation but not necessarily abelian.

Given a subset $B$ of a group $G$, the {\it list of differences of $B$}
is the multiset $\Delta B$ of all possible differences between two distinct elements of $B$.
A collection $\cal F$ of subsets of $G$ is a {\it difference family} (DF) of index $\lambda$ if the multiset sum
$\Delta{\cal F}:=\displaystyle\biguplus_{B\in{\cal F}}\Delta B$ covers every non-zero element of $G$ 
exactly $\lambda$ times. In particular, one says that ${\cal F}$ is a $(v,k,\lambda)$-DF if $G$ has 
order $v$ and its members ({\it base blocks}) have all size $k$.
A difference family is said to be {\it disjoint} (DDF) if its blocks are pairwise disjoint. It is a
{\it partitioned difference family} (PDF) if its blocks partition $G$.

Partitioned difference families, introduced by Ding and Yin \cite{DY} for the construction
of {\it optimal constant composition codes}, are also important from the
design theory perspective; for instance, in \cite{BYW} it is shown a strict connection between PDFs having
all blocks of the same size $k$ but one of size $k-1$ and certain {\it resolvable $2$-designs}  (RBIBDs) with block size $k$. It seems that this paper
passed almost completely unnoticed in spite of the fact that it contains many  new RBIBDs a couple of which are particularly remarkable 
since their parameters are new; a $(45,5,2)$-RBIBD and a $(175,7,2)$-RBIBD.
The importance of the former was pointed out in the paper itself, here we underline
the importance of the latter considering that, according to Table 7.40 in \cite{AGY}, even the
existence of a $(175,7,6)$-RBIBD was previously in doubt while it is now obvious that it can be obtained 
by simply tripling the obtained $(175,7,2)$-RBIBD.

After the above discussed paper the notion of a PDF apparently disappeared from the literature for a long time but, as a matter of fact, it has been considered
under the name of a {\it zero difference balanced function} (ZDBF). A  function $f$ from a group $G$ to a group $H$ is 
defined to be a $(v,\lambda)$-ZDBF if $ord(G)=v$ and the equation $f(g+x)=f(x)$ in the unknown $x$ has always $\lambda$ solutions whichever is $g\in G\setminus\{0\}$.
It is an easy exercise to prove that this is completely equivalent to say that the set of non-empty {\it fibers} of $f$ form a PDF in $G$. 
It seems to be usual to assume that both $G$ and $H$ are abelian (see, e.g., \cite{WZ} and \cite{ZTWY}) but, in our opinion, there is no good reason to make this restriction.


Very recently PDFs have returned with their original name in a paper  \cite{LWG} where the authors develop the 
composition constructions of \cite{BYW} making use of {\it difference matrices}.

It is clear that every DDF can be ``completed" to  a PDF by adding suitable blocks of size 1.
We note that a $(v,k,\lambda)$-DDF necessarily has $1\leq \lambda\leq k-1$ apart the very trivial
case of a $(k,k,k)$ difference set.
The existence of $(v,k,1)$-DDFs is in general a quite hard problem. Among the few results on this problem 
we recall that Dinitz and Rodney \cite{DR} found a $(v,3,1)$-DDF for any admissible $v$ and that 
any {\it radical} $(v,k,1)$-DF (see \cite{Bpairwise}) is disjoint when $k$ is odd.
On the contrary, the literature on $(v,k,k-1)$-DDFs is quite rich and our main purpose is to
get this literature in order.

First of all it is worth mentioning that the $G$-orbit of any  a $(v,k,k-1)$-DDF in $G$ is the {\it near resolution}
of a $(v, k, k-1)$ {\it near resolvable design} (NRB for short). We refer to \cite{AGY} for general background on NRBs.

In the core of this paper we restate and reprove in terms of differences an old nearring theory result - probably passed almost unnoticed by other theorists -
which, starting from {\it Ferrero pairs}, implicitly give a wide class of $(v,k,k-1)$-DDFs in the kernel of a {\it Frobenius group}.

We will show that several constructions for $(v,k,k-1)$-DDFs obtained over the years could be quickly 
obtained as a corollary of that result. In order to further appreciate its effectiveness we apply it to get 
some new DDFs, in particular the Pisano $(p^4,k,k-1)$-DDFs in $\Z_{p^2}\times\Z_{p^2}$ which
arise from the Fibonacci sequence. In the last section we give an example of an infinite class of non-abelian DDFs obtainable 
via the Ferrero construction and, more importantly, we will prove that if all prime factors of $v$ are congruent to 1 modulo $k$,
then there exists a $(v,k,k-1)$-DDF in any group of order $v$.

\section{Some known results}

%
 %

A $(v,k,k-1)$-DDF in $G$ is known in each of the following cases.
\begin{itemize}
\item[(i)] $v$ is odd, $k=2$, and $G$ is any group of order $v$. 
\item[(ii)] $v\equiv1$ (mod 4), $k=4$, $G=\Z_{v}$ and there exists a {\it $\Z$-cyclic whist tournament} on $v$ players $($briefly $Wh(v))$.
\item[(iii)] $v\equiv1$ (mod $k$) is a prime power and $G$ is the additive group of $\mathbb{F}_v$ (the field  of order $v$).
\item[(iv)] The maximal prime power divisors of $v$ are all congruent to 1 (mod $k$)
and $G$ is a direct product of elementary abelian groups.
\item[(v)] 
All prime factors of $v$ are congruent to 1 (mod $k$) and $G=\mathbb{Z}_v$.
\end{itemize}

A DDF as in (i) is nothing but a  {\it starter} of $G$ (see, e.g., \cite{D}). 

The reason of (ii) is that the {\it initial round} of a $\Z$-cyclic $Wh(v)$ 
is a $(v,4,3)$-DDF (but the converse is not generally true). For general background on $\Z$-cyclic whist tournaments we refer to \cite{AF}.

For a DDF as in (iii) - which a special case of the DDFs in (iv) - one can simply take the set of all cosets of the $k$-th roots of unity in the multiplicative group of $\mathbb{F}_v$. 
This DDF is usually attributed to Wilson \cite{W} but we note that it was given earlier in an equivalent form by Ferrero \cite{F}.
We also note that this DDF can be presented as the $(v,k-1)$-ZDBF mapping any $x\in \mathbb{F}_{v}$ into $x^{(v-1)/k}\in  \mathbb{F}_{v}$.

A large class of DDFs as in (iii) and the additional condition that $k\equiv2$ (mod 4) have been very recently given by Li \cite{L}; 
each block of these DDFs is a suitable union of two cosets of the ${k\over2}$-th roots of unity in $\mathbb{F}_v$. 

Results (iv) and (v) have been recently described using the language of ZDBFs and obtained with quite involved proofs in 
\cite{CZHTY} and \cite{DWX}, respectively.
We note, however, that a simple proof of (iv) was given by Furino in 1991 (see end of section 3 in \cite{Fu}) and that
the same proof was implicitly given by Boykett in 2001 (see the proof of Proposition 7 in \cite{B}). 
There is another approach for getting (iv) very quickly from (ii) with the use of {\it difference matrices}; this has been 
very recently noted by Li, Wei and Ge \cite{LWG} but traces of the same approach 
can be found in a very old paper by Jungnickel  (see Corollary 4.5 in \cite{J}).

We note that difference matrices would also allow to obtain (v) very quickly from the existence of a cyclic $(p^n,k,k-1)$-DDF for every  
prime $p\equiv1$ (mod $k$). Such a difference family was obtained by Furino (see Lemma 4.3 in \cite{F}) and it is also deducible
from an even earlier result by Phelps \cite{P} (see Theorem 4.6 in the paper \cite{B1} by the present author).

In the next section we will give a clean construction for  $(v,k,k-1)$-DDFs in the kernel of a Frobenius group
which can be deduced from an old nearring result.
We will show how (iv) and (v) can be almost immediately obtained as special cases of this construction.
In the last section we will prove that (v) can be generalized to any group $G$ of order $v$.

\section{Ferrero $(v,k,k-1)$ difference families}
A Frobenius group $F$ is a semidirect product $G\rtimes A$ with $A$ a non-trivial group of automorphisms of $G$
acting {\it semiregularly} on $G\setminus\{0\}$ (one also says that $A$ is {\it fixed point free}). 
This means that for $g\in G$ and $\alpha\in A$ we have $\alpha(g)=g$ if and only if either $\alpha=id_G$ or $g=0$.
The groups $G$ and $A$ are said to be the {\it kernel} and the {\it complement} of $F$, respectively. 
Any such pair $(G,A)$ is called a {\it Ferrero pair} by {\it nearring} theorists \cite{B, C}.
Note that if $(G,A)$ is a Ferrero pair, then $(H,B)$ is a Ferrero pair as well for every non-trivial subgroup $H$ of $G$ and any
non-trivial subgroup $B$ of $A$.
For general background on Frobenius groups we refer to \cite{I}.

The following theorem is the highlight of this section but we point out that the first part of the theorem is not really new since it is essentially the same as Theorem 5.5 in \cite{C}.
The crucial diversity is that our presentation and proof are given in terms of differences. In the original statement it is said that
any Ferrero pair $(G,A)$ with $ord(G)=v$ and $ord(A)=k$ generates a 2-$(v,k,k-1)$ design and then in the proof it is shown that the blocks 
of this design are all the translates of the $A$-orbits of the non-zero elements of $G$ under the natural action of $G$. 
For the main subject of the present paper, the crucial fact is that the set of $A$-orbits on $G\setminus\{0\}$ is a $(v,k,k-1)$-DDF.

\begin{theorem}\label{Frob}
If $(G,A)$ is a Ferrero pair with $ord(G)=v$ and $ord(A)=k$, then the set of $A$-orbits on $G\setminus\{0\}$
is a $(v,k,k-1)$-DDF in $G$. In the hypothesis that $G$ is abelian and that $vk$ is odd, this DDF is splittable into two $(v,k,{k-1\over2})$-DDFs.
\end{theorem}


\begin{proof}By definition, $A$ acts semiregularly on $G\setminus\{0\}$. This easily implies that
each $A$-orbit on $G\setminus\{0\}$ has size $k$ and that the map $g\in G \longrightarrow \alpha(g)-g\in G$
is a bijection for every $\alpha\in A\setminus\{id_G\}$. Thus we can write:
\begin{equation}\label{ortom}
\{\alpha(g)-g \ | \ g\in G\setminus\{0\}\}=G\setminus\{0\} \quad \forall \alpha\in A\setminus\{id_G\}
\end{equation} 

Let $X$ be a complete set of representatives for the $A$-orbits on $G\setminus\{0\}$ and for each $x\in X$ let
$Orb(x)$ be the $A$-orbit of $x$.
We have to prove that ${\cal O}:=\{Orb(x) \ | \ x\in X\}$ is a $(v,k,k-1)$-DDF. It is evident that the ordered pairs of distinct elements of $Orb(x)$ 
with a fixed second coordinate $y$ are exactly those of the form $(\alpha(y),y)$ with $\alpha\in A\setminus\{id_G\}$. Thus we have
$\Delta Orb(x)=\biguplus_{y\in{\cal O}(x)}\{\alpha(y)-y \ | \ \alpha\in A\setminus\{id_G\}\}$, hence 
$$\Delta{\cal O}=\biguplus_{x\in X}\biguplus_{y\in{\cal O}(x)}\{\alpha(y)-y \ | \ \alpha\in A\setminus\{id_G\}\}=
\biguplus_{g\in G\setminus\{0\}}\{\alpha(g)-g \ | \ \alpha\in A\setminus\{id_G\}\}.$$
So we  can write $\displaystyle\Delta{\cal O}=\biguplus_{\alpha\in A\setminus\{id_G\}}\{\alpha(g)-g \ | \ g\in G\setminus\{0\}\}$ which,
by (\ref{ortom}), is the union of $k-1$ copies of $G\setminus\{0\}$.
The first part of the assertion follows.

From now we assume that $kv$ is odd and that $G$ is abelian.
Suppose that two opposite elements $g$ and $-g$ are in the same $A$-orbit.
In this case there is an $\alpha\in A$ such that $\alpha(g)=-g$.
Then, by induction, we would have $\alpha^i(g)=g$ or $-g$ according to whether $i$ is even or odd, respectively.
Thus, in particular, we would have $\alpha^k(g)=-g$. On the other hand $k$ is the order of $A$ so that $\alpha^k=id_G$.
It follows that $g=-g$. So, considering that $G$ does not have involutions since $v$ is odd, we necessarily 
have $g=0$. We conclude that the set $X$ considered in the first part of our proof can be chosen of the form
$X=Y \cup \ -Y$ with $Y$ a suitable ${k-1\over2}$-subset of $G$. In this way we have that $\cal O$ is splittable into
the two parts ${\cal O}_1=\{Orb(y) \ | \ y\in Y\}$ and ${\cal O}_2=\{Orb(-y) \ | \ y\in Y\}$. Now note that $Orb(-y)=-Orb(y)$.
Hence we obviously have $\Delta Orb(-y)=\Delta Orb(y)$ for each $y\in Y$ since $G$ is abelian. We conclude that
$\Delta{\cal O}_1=\Delta{\cal O}_2$ and then, recalling that $\cal O$ is a $(v,k,k-1)$-DDF, we deduce that both
${\cal O}_1$ and ${\cal O}_2$ are $(v,k,{k-1\over2})$-DDFs.
\end{proof}

The $(v,k,k-1)$-DDFs produced by the above theorem will be said {\it Ferrero difference families}. 

Note that the {\it patterned starter} of a group $G$ of odd order (namely the set of all possible pairs
$\{g,-g\}$ of opposite elements of $G\setminus\{0\}$) can be seen as the Ferrero DF 
determined by the Ferrero pair $(G,\{id,-id\})$ when $G$ is abelian.

As a first immediate consequence of Theorem \ref{Frob} we have the following result which
was also stated in a weaker form by Furino (\cite{Fu}, Lemma 4.2).

\begin{lemma}\label{ring}
Let $R$ be a ring of order $v$ with unity, and let $U(R)$ be the group of units of $R$. If $U$ is a subgroup of order $k$ of $U(R)$
with $u-1\in U(R)$ for each $u\in U\setminus\{1\}$, then there exists a Ferrero $(v,k,k-1)$-DDF in the additive group of $R$.
\end{lemma}
\begin{proof} Any subgroup $U$ of $U(R)$ can be seen as an automorphism group of the additive group $G$ of $R$.  
Indeed any $u\in U(R)$ can be identified with the automorphism of  $G$ mapping $x$ into $ux$.
It is also clear that $u-1\in U(R)$ for each $u\in U\setminus\{1\}$ implies that $U$ acts semiregularly on $G\setminus\{0\}$.
The assertion then follows from Theorem \ref{Frob}.\end{proof}

Now we show how  the previous lemma allows to obtain result (iv) very quickly. We essentially give the same old easy proof given by Furino \cite{Fu},
not comparable to the recent tortuous proof in \cite{DWX}. 
\begin{corollary}\label{EA}
Let $v$ be a product of prime powers $q_1,\dots, q_n$ all congruent to $1$ $($mod $k)$. 
Then there exists a Ferrero $(v,k,k-1)$-DDF in the additive group of $\mathbb{F}_{q_1}\times\dots\times \mathbb{F}_{q_n}$.
\end{corollary}
\begin{proof}For $1\leq i\le n$, take a $k$-th primitive root $u_i$ of $\mathbb{F}_{q_i}$.
It is immediate that  $u:=(u_1,\dots,u_n)$ is a unit of order $k$ of the ring 
$R=\mathbb{F}_{q_1}\times\dots\times\mathbb{F}_{q_n}$ and that $u^i-1$ is a unit of $R$ for $1\leq i\leq k-1$.
The assertion then follows from Lemma \ref{ring}.\end{proof}

Let us say that a Ferrero pair $(G,A)$ has parameters $(v,k)$ if $v$ and $k$ are the orders of $G$ and $A$,
respectively. A trivial necessary condition for $(v,k)$ to be the parameters of a suitable Ferrero pair is that
$k$ divides $v-1$. From the above corollary one deduces that a sufficient condition is that $q\equiv1$ (mod $k$) 
for every maximal prime power factor $q$ of $v$. This condition has been proved to be also necessary 
by Boykett (\cite{B}, Corollary 6) as a consequence of other results using, in particular, Thompson's 
theorem on Frobenius groups. We reprove  this below in a simpler and more direct way. 
\begin{proposition}
There exists a suitable Ferrero pair of parameters $(v,k)$ - or equivalently a Ferrero $(v,k,k-1)$-DDF in a suitable group -
if and only if 
$q\equiv1$ $($mod $k)$ for every maximal prime power factor $q$ of $v$.
\end{proposition}
\begin{proof} 
$(\Longrightarrow)$ Let $(G,A)$ be a Ferrero pair of parameters $(v,k)$ and let $p^e$, $p$ prime, be a 
maximal prime power factor of $v$.
The group $A$ acts as a permutation group on the set $Syl_p(G)$ of all Sylow $p$-subgroups of $G$.
If $A$ does not fix any member of $Syl_p(G)$, then each $A$-orbit on 
$Syl_p(G)$ would have size $|A|=k$ so that $k$ divides $|Syl_p(G)|$. In its turn $|Syl_p(G)|$ divides ${v\over p^e}$
by the third Sylow theorem. We conclude that $k$ divides both $v-1$ and ${v\over p^e}$, hence $k=1$
which is absurd. So $A$ fixes a suitable $S\in Syl_p(G)$. It follows that $(S,A)$ is a Ferrero pair so that 
$ord(A)=k$ divides $ord(S)-1=p^e-1$ which is the assertion. 

$(\Longleftarrow)$ See Corollary \ref{EA}.
\end{proof}

Now we show how result (v) can be generalized to any abelian group and that
this can be also quickly obtainable as a corollary of Lemma \ref{ring}.

\begin{corollary}\label{cyclic}
If all prime factors of $v$ are congruent to $1$ $($mod $k)$,
then there exists a Ferrero $(v,k,k-1)$-DDF in any abelian group $G$ of order $v$.
\end{corollary}
\begin{proof}First recall that if $q=p^e$ with $p$ prime, then 
$U(\mathbb{Z}_{q})$ is cyclic of order $\phi(q)=p^{e-1}(p-1)$ (see, e.g., \cite{A}). 
Also, the subgroup $S_q$ of $U(\mathbb{Z}_{q})$ of order $p^{e-1}$ consists of all elements $s\in \mathbb{Z}_q$ with $s\equiv1$ (mod $p$).

Let $G$ be an abelian group of order $v$. By the Fundamental Theorem of Finite Abelian Groups, there are suitable prime powers $q_1,\dots, q_n$ dividing $v$ such that,
up to isomorphism, $G$ is the additive group of the ring $R=\mathbb{Z}_{q_1}\times\dots\times\mathbb{Z}_{q_n}$. Of course we have
$U(R)=U(\mathbb{Z}_{q_1})\times\dots \times U(\mathbb{Z}_{q_n})$. 

For $1\leq i\leq n$, set $q_i=p_i^{e_i}$ with $p_i$ prime. By assumption, $k$ divides $p_i-1$, hence $U(\mathbb{Z}_{q_i})$ has an element $u_i$ of order $k$. 
Obviously $\langle u_i\rangle$ has trivial intersection with $S_{q_i}$, the subgroup of $U(\mathbb{Z}_{q_i})$ of order $p_i^{e_i-1}$.
Hence, for $1\leq i\leq n$ and $1\leq j\leq k-1$, we have $u_i^j\notin S_{q_i}$, i.e., $u_i^j\not\equiv1$ (mod $p$).

We conclude that  $u:=(u_1,\dots,u_n)$ is a unit of  $R$ of order $k$ and that $u^j-1$ is also a unit of $R$ for $1\leq j\leq k-1$. 
The assertion then follows from Lemma \ref{ring}.\end{proof}

There are infinite classes of Ferrero DDFs which neither Corollary  \ref{EA} nor Corollary \ref{cyclic}
are able to capture. One of these classes will be given in the next section.

Here we only give an easy example of a Ferrero $(q^4,3,2)$-DDF in $\mathbb{Z}_{q^2}\times \mathbb{Z}_{q^2}$
for any prime power $q$ not divisible by 3. Such a DDF cannot be obtained from Corollary \ref{cyclic} when $q$ is a
power of 2 or the power of a prime $p\equiv5$ (mod 6).
Consider the automorphism $\alpha$  of $\mathbb{Z}_{q^2}\times \mathbb{Z}_{q^2}$
defined by $\alpha(x,y)=(y-x,-x)$. One can see that the group $A$ generated by $\alpha$ has order 3 and acts
semiregularly on $G\setminus\{(0,0)\}$. So we obtain the required $(q^4,3,2)$-DDF by applying 
Theorem \ref{Frob}.

For instance, for $q=2$, we get the following Ferrero $(16,3,2)$-DDF in $\mathbb{Z}_4\times \mathbb{Z}_4$
(in order to save space, any
pair $(x,y)$ is denoted by $xy$):
$$\bigl{\{}\{01,10,33\}, \ \{02,20,22\}, \ \{03,30,11\}, \  \{12,13,23\}, \ \{21,32,31\}\bigl{\}}.$$

\section{Ferrero difference families from the Fibonacci sequence}

Here we present an infinite class of Ferrero PDFs obtainable via the Fibonacci sequence which, in many cases, 
cannot be deduced from Corollary  \ref{EA} or Corollary \ref{cyclic}. For this, we need to recall some number theoretic arguments.

The {\it Pisano period modulo $n$}, denoted $\pi(n)$, is the period of the Fibonacci sequence modulo $n$ which is also equal to
the period of the {\it Fibonacci matrix} $\bf F=\begin{pmatrix}1&1\cr1&0\end{pmatrix}$ in the group $GL_2(n)$ of all $2\times2$ invertible 
matrices of the ring $\mathbb{Z}_{n}$. There is no known formula for $\pi(p)$ with $p$ a prime but the following properties have been established so far
(see, e.g., \cite{R}):
\begin{itemize}
\item[$(P_1)$] $\pi(p^2)$ is equal either to $p\pi(p)$ or $\pi(p)$;
\item[$(P_2)$] $\pi(p)$ is equal to the least common multiple of the periods of the two eigenvalues of $\bf F$ in the multiplicative group of $\mathbb{F}_{p^2}$;

\item[$(P_3)$] $\pi(p)=\begin{cases}3 \hfill\mbox{ if $p=2$}
\smallskip\cr20 \hfill\mbox{ if $p=5$}
\smallskip\cr
\mbox{an even divisor of $p-1$}\hfill \mbox{if $p\equiv\pm1$ (mod 10)}
\smallskip\cr
\mbox{${2(p+1)\over d}$ with $d$ an odd divisor of $p+1$}\quad\hfill \mbox{if $p\equiv\pm3$ (mod 10)} \end{cases}$
\end{itemize}

Regarding property $(P_1)$, it should be noted that there is no known prime $p$ for which $\pi(p^2)=\pi(p)$ holds.

\begin{proposition}\label{pisano}
There exists a Ferrero $(p^4,k,k-1)$-DDF in $\mathbb{Z}_{p^2}\times\mathbb{Z}_{p^2}$ for any prime 
$p\neq5$ and any divisor $k$ of the Pisano period $\pi(p)$.
\end{proposition}
\begin{proof} An example of a $(16,3,2)$-DDF in $\mathbb{Z}_{4}\times\mathbb{Z}_{4}$ has been given at the end of the previous section,
therefore the assertion is true for $p=2$. For $p\equiv\pm1$ (mod 10) the assertion is an immediate consequence of Corollary \ref{cyclic}
and property $(P_3)$. So, in the following, we will assume that $p\equiv\pm3$ $($mod $10)$. This implies
that $5$ is not a square in $\mathbb{F}_p$, hence the two eigenvalues $\lambda_1$, $\lambda_2$ of $\bf F$ are ``conjugates" in 
$\mathbb{F}_{p^2}$. Indeed we have $\{\lambda_1,\lambda_2\}=\{{1+\sqrt{5}\over2},{1-\sqrt{5}\over2}\}$. So the $i$-th powers of $\lambda_1$ and $\lambda_2$
are conjugates as well. It follows that $\lambda_1^i=1$ if and only if $\lambda_2^i=1$ which clearly implies that $\lambda_1$ and $\lambda_2$ 
have the same period in  the multiplicative group of $\mathbb{F}_{p^2}$.

Let us identify any matrix $\begin{pmatrix}a&b\cr c&d\end{pmatrix}\in GL_2(\mathbb{Z}_{p^2})$ 
with the automorphism of  $\mathbb{Z}_{p^2}\times\mathbb{Z}_{p^2}$ mapping $(x,y)$ into $(ax+by,cx+dy)$.

The period of the  Fibonacci matrix $\bf F$
in $GL_2(\mathbb{Z}_{p^2})$ is equal to $\pi(p^2)$ which, by property $(P_1)$, is equal either to $p\pi(p)$ or $\pi(p)$. 
Let $A$ be the subgroup of order $\pi(p)$ of the group generated by
$\bf F$. Thus $A=\langle \Phi\rangle$ with $\Phi=\bf F$ or $\Phi= {\bf F}^p$ according to whether $\pi(p^2)=\pi(p)$ or $\pi(p^2)=p\pi(p)$, respectively.

Assume that 1 is an eigenvalue of $\Phi^i$.
Then, considering that for any matrix $M$ and any positive integer $j$ we have $Spec(M^j)=\{\lambda^j \ | \ \lambda\in Spec(M)\}$, we have
\begin{center}
$\begin{cases}
\lambda_1^i=\lambda_2^i=1\quad\,\,\,\,\,\, \mbox{if $\pi(p^2)=\pi(p)$}\cr \lambda_1^{pi}=\lambda_2^{pi}=1\quad \mbox{ if $\pi(p^2)=p\pi(p)$}
\end{cases}$
\end{center}

Thus, by property $(P_2)$, we have that $\pi(p)$ divides $i$ in the former case while $\pi(p)$ divides $pi$ in the latter. 
Anyway $\pi(p)$ and $p$ are coprimes by property $(P_3)$ so that, in both cases, $i$ should be divisible by $\pi(p)$. 
Recalling that $\pi(p)$ is the order of $\Phi$, we conclude that $\Phi^i$ is the identity matrix. 

Now note that a matrix $M\in GL_2(\mathbb{Z}_{p^2})$ is fixed point free if and only if 1 does not belong to $Spec(M)$.
So we have proved that the group $A$ acts semiregularly on $\mathbb{Z}_{p^2}\times \mathbb{Z}_{p^2}\setminus\{(0,0)\}$,
i.e., $(\mathbb{Z}_{p^2}\times \mathbb{Z}_{p^2},A)$ is a Ferrero pair. Of course $(\mathbb{Z}_{p^2}\times \mathbb{Z}_{p^2},\langle \Phi^{\pi(p)/k}\rangle)$
is a Ferrero pair as well for each divisor $k$ of $\pi(p)$ and then the assertion follows from Theorem \ref{Frob}.
\end{proof}

The {\it Pisano DDFs}, namely the $(p^4,k,k-1)$-DDFs which are built as in the proof of the above proposition, 
allow to largely enrich the set of known values of $k$ for which there exists a $(p^4,k,k-1)$-DDF in  $\mathbb{Z}_{p^2}\times \mathbb{Z}_{p^2}$
in the case of $p\equiv\pm3$ (mod 10).
In particular, for a Mersenne prime $p=2^{4n+3}-1$ we have $p\equiv-3$ (mod 10) and then we necessarily have $\pi(p)=2^{4n+4}$ by property $(P_4)$.
Thus there exists a Pisano $(p^4,2^i,2^i-1)$-DDF in $\mathbb{Z}_{p^2}\times \mathbb{Z}_{p^2}$ for $1\leq i\leq 4n+4$.

\medskip
As an example, let us apply Proposition \ref{pisano} with $p=3$. We have $\pi(3)=6$ and $\pi(3^2)=3\cdot\pi(3)=18$.
Then the matrix $\Phi$ considered in the proof of the above proposition is ${\bf F}^3=\begin{pmatrix}3&2\cr2&1\end{pmatrix}$
and the group $A$ generated by $\Phi$ is the following:
$$\biggl{\{}
\begin{pmatrix}3&2\cr2&1\end{pmatrix},
\begin{pmatrix}4&8\cr8&5\end{pmatrix},
\begin{pmatrix}1&7\cr7&3\end{pmatrix},
\begin{pmatrix}8&0\cr0&8\end{pmatrix},
\begin{pmatrix}6&7\cr7&8\end{pmatrix},
\begin{pmatrix}3&2\cr2&1\end{pmatrix},
\begin{pmatrix}5&1\cr1&4\end{pmatrix},
\begin{pmatrix}1&0\cr0&1\end{pmatrix}
\biggl{\}}$$ 
The ten orbits of $A$ on $\mathbb{Z}_9\times \mathbb{Z}_9\setminus\{(0,0)\}$ are listed below where, again,
any pair $(x,y)$ will be simply denoted by $xy$:
$$B_0=\{21, 85, 73, 08, 78, 14, 26, 01\};\quad B_1=\{42,71,56,07,57,28,43,02\};$$
$$B_2=\{63,66,30,06,36,33,60,03\};\quad B_3=\{04,52,13,05,15,47,86,04\};$$
$$B_4=\{32,48,17,80,67,51,82,10\};\quad B_5=\{53,34,81,88,46,65,18,11\};$$
$$B_6=\{74,20,64,87,25,70,35,12\};\quad B_7=\{68,72,77,83,31,27,22,16\};$$
$$B_8=\{37,54,55,76,62,45,44,23\};\quad B_9=\{58,40,38,75,41,50,61,24\}.$$
Thus the $B_i$s are the blocks of a Pisano $(81,8,7)$-DDF in $\mathbb{Z}_9\times \mathbb{Z}_9$.

\subsection{Non-abelian disjoint $(v,k,k-1)$ difference families}

Here we give some constructions for DDFs in non-abelian groups.
Let us start by giving a class of non-abelian Ferrero DDFs.

\begin{proposition}\label{notabelian}
If $R=(V,+,\cdot)$ is a ring with unity admitting a group $U$ of units 
such that $u^2-1\in U(R)$ for each $u\in U\setminus\{1\}$, then there exists a non-abelian Ferrero $(v^3,k,k-1)$-DDF with $v=|V|$ and $k=|U|$.
\end{proposition} 
\begin{proof}
Let us equip the set $V^3$ with the operation $\oplus$ defined by the rule
$$(x_1,y_1,z_1)\oplus (x_2,y_2,z_2)=(x_1+x_2, \ y_1+y_2, \ z_1+z_2+x_1\cdot y_2).$$
It is an easy exercise to check that $G=(V^3,\oplus)$ is a group. Also note that $G$ is non-abelian
since we have, for instance, $(0,1,0)\oplus(1,0,0)=(1,1,0)$ while $(1,0,0)\oplus(0,1,0)=(1,1,1)$.
Now note that for each $u\in U$, the map $\alpha_u: (x,y,z)\in V^3 \longrightarrow (u\cdot x,u\cdot y,u^2\cdot z)\in V^3$
is an automorphism of $G$ and that $(G,A:=\{\alpha_u \ | \ u\in U\})$ is a Ferrero pair. The assertion then follows
from Theorem \ref{Frob}. 
\end{proof}

We remark that a group $U$ as in the statement of the above proposition is necessarily of odd order.
Indeed, in the opposite case, $U$ would have at least one involution, say $u$, and then $u^2-1=0\notin U(R)$
against the assumption.

Applying Proposition \ref{notabelian} with $R=\mathbb{F}_q$ we obtain the following.

\begin{corollary}
There exists a non-abelian Ferrero $(q^3,k,k-1)$-DDF
for any pair $(q,k)$ with $q$ a prime power and $k$ any odd divisor of $q-1$.
\end{corollary}

Let $\cal F$ be the $(v^3,k,k-1)$-DDF in $(V^3,\oplus)$ obtainable by Proposition \ref{notabelian}. We remark that $\cal F$
actually coincides with the $(v^3,k,k-1)$-DDF in the abelian group of the ring $R\times R\times R$ obtainable using Lemma \ref{ring}. 
On the other hand to consider $\cal F$ as a DDF in $(V^3,\oplus)$ is not the same as to consider $\cal F$ as a DDF in $(V^3,+)$. 
Indeed the $(v^3,k,k-1)$-NRB whose near resolution is the orbit of $\cal F$ under $(V^3,\oplus)$ does not coincide with the $(v^3,k,k-1)$-NRB 
whose near resolution is the orbit of $\cal F$ under $(V^3,+)$.

\medskip
By Corollary \ref{cyclic} there exists a $(v,k,k-1)$-DDF in any abelian group $G$ of order $v$
provided that all primes in $v$ are congruent to 1 (mod $k$). We are going to see that this result
remains true if one removes the hypothesis of commutativity of $G$. The present author proved that the existence 
of a $(p,k,\lambda)$-DF for any prime $p$ in $v$ implies the existence of a
$(v,k,\lambda)$-DF in any group $G$ of order $v$ (see Corollary 5.5 in \cite{B1}). Now we reprove 
this theorem showing that if all component $(p,k,\lambda)$-DFs are disjoint, then the resultant 
$(v,k,\lambda)$-DF in $G$ is disjoint as well.

\begin{theorem}\label{composingDDFs}
If $G$ is a group of order $v$ and there exists a $(p,k,\lambda)$-DF (resp. DDF) for every prime factor $p$ of $v$,
then there exists a $(v,k,\lambda)$-DF (resp. DDF) in $G$. 
\end{theorem}
\begin{proof}
The cases $k=1$ and $k=2$ are trivial. So, in the following, we assume $k>2$.
We prove the theorem by induction on $v$. The assertion is trivially true for $v=1$; in this case the required DF is the empty family. 
Let $G$ be a group of order $v>1$ as in the statement and assume that the assertion is true for all groups of order less than $v$.
First observe that $v$ is necessarily odd since the existence of a $(p,k,\lambda)$-DF for any prime factor $p$ of $v$
implies that $p\geq k>2$ for any such prime $p$.
It follows, by the Feit-Thompson theorem, that $v$ is solvable. Thus, in particular, $G$ has a normal subgroup $N$ of prime index, say $p$.
By hypothesis there exists a $(p,k,\lambda)$-DF (resp. DDF), say ${\cal F}_1$, in $G/N$. By induction,
there also exists a $({v\over p},k,\lambda)$-DF (resp. DDF), say ${\cal F}_2$, in $N$. 
For each block $B=\{g_1+N,\dots,g_k+N\}\in{\cal F}_1$
and any $n\in N$ consider the $k$-subset $B(n)$ of $G$ defined by
$B(n)=\{g_i+in \ | \ 1\leq i\leq k\}$. We claim that $${\cal F}:=\{B(n) \ | \ B\in{\cal F}_1, n\in N\} \ \cup \ {\cal F}_2$$
is a $(v,k,\lambda)$-DF (resp. DDF) in $G$.

Given $g\in G\setminus N$, let $g+N=(g_i+N)-(g_j+N)$ be a representation of $g+N$ as a difference from 
a block $B=\{g_1+N,\dots,g_k+N\}$ of ${\cal F}_1$.
Consider the element $n:=-g_i+g+g_j$, necessarily belonging to $N$, and let $ord(n)$ be its order.
We have $1\leq |i-j|<k$, hence $i-j$ is coprime with $ord(n)$ since, by assumption,
every divisor of $ord(G)$ distinct by 1 is clearly greater than $k$. Thus there exists the inverse, say $x$, of $i-j$ modulo $ord(n)$.
Now check that $g$ is the difference between the $i$-th element and the $j$-th element of the block $B(xn)$ of ${\cal F}_1$:
$$(g_i+ixn)-(g_j+jxn)=g_i+(i-j)xn-g_j=g_i+n-g_j=g_i-g_i+g+g_j-g_j=g.$$
In this way we have proved that each of the $\lambda$ representations of $g+N$ as a difference from ${\cal F}_1$
leads to a representation of $g$ as a difference from ${\cal F}$. Thus every element of $G\setminus N$ is covered at
least $\lambda$ times by $\Delta{\cal F}$. The same is true for all elements of $N\setminus\{0\}$ since 
$\Delta{\cal F}_2$ is $\lambda$ times $N\setminus\{0\}$.
Now note that the number of blocks of ${\cal F}$ is given by 
\begin{center}
$|{\cal F}|=|{\cal F}_1|\cdot|N|+|{\cal F}_2|={\lambda(p-1)\over k(k-1)}\cdot{v\over p}+{\lambda\over k(k-1)}({v\over p}-1)={\lambda(v-1)\over k(k-1)}$
\end{center}
and then $|\Delta{\cal F}|=k(k-1)|{\cal F}|=\lambda(v-1)$. It follows, by the pigeon hole principle, that 
every non-zero element of $G$ is covered by $\Delta{\cal F}$ exactly $\lambda$ times, i.e., ${\cal F}$ is a 
$(v,k,\lambda)$-DF. 

It remains to prove that ${\cal F}$ is disjoint in the hypothesis that both ${\cal F}_1$ and ${\cal F}_2$ are disjoint. 

Every block of the form $B(n)$ with $B=\{g_1+N,\dots,g_k+N\}\in{\cal F}_1$ and $n\in N$ has no element in $N$ 
otherwise we would have $g_i+in\in N$ for some $i$, hence $g_i+N=N$
contradicting the fact that the blocks of ${\cal F}_1$ partition $G/N\setminus\{N\}$.
Thus $B(n)$ is disjoint with every block of ${\cal F}_2$.

Now assume that $B(n)$ and $B'(n')$ have an element in common for some blocks $B=\{g_1+N,\dots,g_k+N\}$ and
$B'=\{g'_1+N,\dots,g'_k+N\}$ of ${\cal F}_1$ and some elements $n$, $n'$ of $N$.
Thus we have $g_i+in=g'_j+jn'$ for suitable $i, j\in\{1,\dots,k\}$. This implies that $g_i+N=g'_j+N$,
hence $B=B'$ and $i=j$ since ${\cal F}_1$ is disjoint. It follows that $i(n-n')=0$, hence $ord(n-n')$ is a divisor of $i$
which implies $ord(n-n')=1$ since every divisor of $ord(G)$ distinct by 1 is greater than $k$. We conclude that 
$B=B'$ and $n=n'$, i.e., $B(n)=B'(n')$.

Finally any two distinct blocks of ${\cal F}_2$ are disjoint by assumption. The assertion follows.
\end{proof}

We are now finally able to prove the main result of this section.

\begin{corollary}
If all prime divisors of $v$ are congruent to $1$ $($mod $k)$, then there exists a $(v,k,k-1)$-DDF and a  $(v,k,{k-1\over2})$-DDF in $G$
for any group $G$ of order $v$.
\end{corollary}
\begin{proof}
We know that there exists a $(p,k,k-1)$-DDF and a $(p,k,{k-1\over2})$ for any prime $p\equiv1$ (mod $k$). Then the assertion immediately follows from Theorem \ref{composingDDFs}.
\end{proof}

\normalsize
\section*{Acknowledgement}
This work has been performed under the auspices of the G.N.S.A.G.A. of the
C.N.R. (National Research Council) of Italy.






\begin{thebibliography}{99}

\bibitem{AGY} R.J.R. Abel, G. Ge and J. Yin,
{\it Resolvable and Near-Resolvable Designs},
Handbook of Combinatorial Designs, Second Edition, C.J. Colbourn and J.H. Dinitz (Editors), Chapman \& Hall/CRC, Boca Raton, FL, 2006, 124--132.

\bibitem{AF}
I. Anderson and N.J. Finizio, {\it Whist tournaments}, 
Handbook of Combinatorial Designs, Second Edition, C.J. Colbourn and J.H. Dinitz (Editors), Chapman \& Hall/CRC, Boca Raton, FL, 2006, 663--668.

\bibitem{A}
T. Apostol, Introduction to Analytical Number Theory. Springer Verlag, 1976

\bibitem{B} T. Boykett,
{\it Constructions of Ferrero pairs of all possible orders},
SIAM Discrete Math. {\bf14} (2001), 283--285.

\bibitem{B1} M. Buratti {\it Recursive constructions for difference matrices 
and relative difference families}, J. Combin. Des. {\bf 6} (1998), 165--182.

\bibitem{Bpairwise}
M. Buratti,
{\it Pairwise balanced designs from finite fields}, Discrete Math. {\bf208/209} (1999), 103--117.


\bibitem{BYW} M. Buratti, J. Yan and C. Wang,
{\it From a $1$-rotational RBIBD to a partitioned difference family},
Electronic J. Combin. {\bf17} (2010), $\sharp$R139.

\bibitem{C}
J.R. Clay, Nearrings: Geneses and Applications. Oxford University Press, Oxford, UK, 1992.

\bibitem{CZHTY}
H. Cai, X. Zeng, T. Helleseth, X. Tang, and Y. Yang, 
{\it A new construction of zero-difference balanced functions and its applications},
IEEE Trans. Inf. Theory {\bf59} (2013), 5008--5015.

\bibitem{DY} C. Ding and J. Yin, 
{\it Combinatorial Constructions of Optimal Constant Composition Codes}, 
IEEE Trans. Inform. Theory {\bf51} (2005), 3671--3674.

\bibitem{DWX}
C. Ding, Q. Wang and M. Xiong, {\it Three new families of zero-difference balanced functions with applications},
IEEE Trans. Inf. Theory {\bf60} (2014), 2407--2413.

\bibitem{D}
J.H. Dinitz, {\it Starters},
Handbook of Combinatorial Designs,
Second Edition,  C.J. Colbourn and J.H. Dinitz (Editors), Chapman \&
Hall/CRC, Boca Raton, FL, 2006, 622-628.

\bibitem{DR}
J.H. Dinitz and P. Rodney, {\it Block disjoint difference families for Steiner triple systems}. Utilitas Math. {\bf52} (1997), 153--160.


\bibitem{F}
G. Ferrero, {\it Stems planari e BIB-disegni}, Riv. Math. Univ. Parma {\bf11} (1970), 79--96.

\bibitem{Fu}
S. Furino, {\it Difference families from rings}, Discrete Math. {\bf97} (1991), 177--190.

\bibitem{I}
I.M. Isaacs, {Finite Group Theory}. Graduate Studies in Mathematics, vol. 92. American Mathematical Society, 2008.

\bibitem{J}
D. Jungnickel, {\it Composition theorems for difference families and regular planes}, Discrete Math. {\bf23} (1978), 151--158.


\bibitem{L}
L. Li, {\it A note on difference families from cyclotomy}, to appear in Discrete Math.

\bibitem{LWG} S. Li, H. Wei and G. Ge,
{\it Generic constructions for partitioned difference families with applications: a unified combinatorial approach},
Des. Codes Cryptogr. {\bf82} (2017), 583--599.

\bibitem{P}
K. T. Phelps, {\it Isomorphism problems for cyclic block designs}, Ann. Discrete Math. {\bf34} (1987), 385--392


\bibitem{R} M. Renault,
{\it The period, rank, and order of the $(a, b)$-Fibonacci sequence mod $m$},
Math. Mag. {\bf 86} (2013), 372--380.


\bibitem{WZ} Q. Wang and Y. Zhou,
{\it Sets of zero-difference balanced functions and their applications},
Adv. Math. Commun. {\bf8} (2014), 83--101.

\bibitem{W}
R.M. Wilson, {\it Cyclotomic and difference families in elementary abelian groups}, 
J. Number Theory 4 (1972), 17--47.

\bibitem{ZTWY} Z. Zhou, X. Tang, D. Wu, and Y. Yang, {\it Some new classes of zero
difference balanced functions}, IEEE Trans. Inf. Theory {\bf58} (2012), 139--145.

\end{thebibliography}
\end{document}